\documentclass{article}
\usepackage{maa-monthly-noheader}

\theoremstyle{theorem}
\newtheorem{theorem}{Theorem}
\newtheorem{corollary}{Corollary}
\newtheorem{lemma}{Lemma}

\theoremstyle{definition}

\newtheorem*{remark}{Remark}

\newcommand{\N}{\mathbb N}

\newcommand{\FLT}{\mathrm{FLT}}
\newcommand{\DM}{\mathrm{DM}}
\newcommand{\IRT}{\mathrm{IRT}}
\begin{document}

\title{
Fermat's Last Theorem Implies\\ Euclid's Infinitude of Primes}
\markright{Fermat's Last Theorem guarantees primes}
\author{Christian Elsholtz
}
\maketitle

\begin{abstract}
We show that Fermat's last theorem and a combinatorial theorem of Schur on 
monochromatic solutions of $a+b=c$ implies that there exist infinitely 
many primes. In particular, for small exponents such as $n=3$ or $4$ 
this gives a new proof of Euclid's theorem, 
as in this case Fermat's last theorem has a proof that does not use 
the infinitude of primes. Similarly, we discuss implications of
Roth's theorem on arithmetic progressions, Hindman's theorem, and infinite
Ramsey theory towards Euclid's theorem.
As a consequence we see that Euclid's Theorem is a necessary condition for many
interesting (seemingly unrelated) results in mathematics.
\end{abstract}

\section{Introduction.}
Imagine that the set of positive integers has only finitely many primes.
We will investigate consequences, and to become more creative with this, we imagine 
we live in an entirely different world, namely in a ``world with only finitely many primes.''
If you are a number theorist, then you will realize that a major part of analytic number theory
just vanishes. One of the implications of this article is that algebraic
number theorists and combinatorialists would live in a very different 
world, too.
The reason is that ``Fermat's last theorem'' even in the first interesting case 
with exponent $3$ would be wrong, that major parts of the modern subject of
additive combinatorics would disappear,
and that even basic results of infinite Ramsey theory would not exist. 
If you wonder why
this is the case, we invite you to a journey
of unexpected discoveries in the 
fictional ``world with only finitely many primes''!

 There are many proofs of Euclid's theorem stating that there exist
infinitely many primes. 
There is a very thorough bibliographic collection of 70 pages on a multitude of proofs of Euclid's theorem,
due to Me\v{s}trovi\'{c} \cite{Mestrovic}. 
Other collections are given by Ribenboim \cite{Ribenboim} 
and a very recent one by Granville \cite{GranvilleLMS}. 
For some recent proofs, see \cite{Saidak:2006, Wooley}.

Many of these proofs make use of an infinite sequence
with mutually coprime integers, such as $F_n=2^{2^n}+1$ (Goldbach, in a letter to Euler
1730), or primitive divisors of certain recursive sequences 
(see, e.g., \cite{Saidak:2006}).
Furstenberg \cite{Furstenberg} made use of a suitably defined topology to prove
Euclid's theorem.
A number of proofs have used the exponents of a prime factorization;
see, for example,
\cite{Elsholtz:2012, Erdos:1938, Polya:1918}.
Even more recently, two proofs \cite{Alpoge:2015, Granville}
made use of van der Waerden's theorem
applied to the \emph{patterns of exponents}.
Alpoge \cite{Alpoge:2015} introduced van der Waerden's theorem to this subject,
and Granville \cite{Granville} combined Alpoge's idea with
a theorem of Fermat, namely that there are no 
four squares in arithmetic progression.

Inspired by this new type of 
proof, we investigate which type of purely combinatorial results 
can be combined with some kind of arithmetic result to give new proofs that
there exist infinitely many primes. 
In this way we link Euclid's theorem to some very beautiful and significant 
results of modern mathematics.

Here is a brief outline of the article.
In Section {\ref{sec:Fermat}} we link Euclid's theorem
to Fermat's last theorem,  eventually 
proved by Wiles \cite{Wiles:1995}, and to a theorem of Schur (1916),
which is often considered to be 
the starting point of combinatorial number theory.\\
In Section {\ref{sec:Roth}}
the link is to a theorem of Roth (1953) on the density of integers
without arithmetic progressions.
An independent elementary proof of Euclid's theorem is a by-product, in Section
{\ref{sec:density}}.
Section {\ref{sec:discussion}} has some discussion about varying the
number-theoretic or combinatorial input.
Section {\ref{sec:Hindman}} uses
a theorem of Hindman (1974) on an infinite extension of
Schur's theorem, and Section {\ref{sec:IRT}} gives two proofs using
infinite Ramsey theory.

Roth's theorem and its extension by Szemer\'{e}di
\cite{Szemeredi:1975}, and 
quantitative versions thereof, 
(e.g., due to Bourgain \cite{Bourgain:2008}, Gowers \cite{Gowers:2001}, 
Green and Tao \cite{GreenandTao:2017})
has inspired many excellent mathematicians
 and has had tremendous impact on 
the relatively young field of additive combinatorics.

\section{Fermat's last theorem implies Euclid's theorem.}
{\label{sec:Fermat}}
We first state Schur's theorem and then the main result of this article.
\begin{lemma}[Schur's theorem \cite{Schur}, 1916]
For every positive integer $t$, there exists an integer $s_t$ such that
if one colors each integer $m \in[1,s_t]$ using one of $t$ distinct colors, 
then there is a monochromatic solution of $a+b=c,$ where $a,b,c \in [1,s_t]$.
\end{lemma}

\begin{theorem}{\label{thm:fermat-euclid}}
For $n \geq 3$ let $\FLT(n)$ denote the statement
``There are no solutions of the equation 
$x^n+y^n=z^n$ in positive integers $x,y,z$.''
Then ``$\FLT(n)$ is true'' and Schur's theorem imply that
there exist infinitely many primes.
\end{theorem}
Theorem \ref{thm:fermat-euclid} 
gives a new proof of Euclid's theorem for those exponents
for which a proof of $\FLT(n)$ independent of the infinitude of
primes exists.
This is certainly the case for $n=3,4,5$, 
where elementary proofs exist (see 
\cite{Edwards, Ribenboim-Fermatbook}).
It then trivially follows for
infinitely many exponents, for example for all multiples of $3$.
 The application
of Fermat's last theorem with general $n$ to Euclid's theorem
might possibly compete for the most indirect proof, 
but at present the proof with
general $n$ is not
actually a proof at all,
as Wiles's proof makes use of the fact that there exist infinitely many primes.

We briefly show that Schur's theorem nowadays can be seen as a 
direct consequence of Ramsey's theorem \cite{Ramsey} (1929).
Ramsey's theorem (see \cite[Theorem 10.3.1]{Cameron}) states that, 
for any number $t$ of colors (let us call them $1, \ldots , t$)
and positive integers $n_1, \ldots, n_t$
there exists an integer $R(n_1, \ldots, n_t)$ such that if the edges of the
complete graph on $R(n_1, \ldots, n_t)$ vertices are colored, there exists
an index $i$ and a monochromatic clique of size $n_i$ all of whose edges are of 
color $i$.
In our application we only need the case $n_1=\cdots =n_t=3$.

Let $\chi: \{1,\ldots, N\} \rightarrow \{1, \ldots , t\}$ be the coloring of the first 
$N=R(3, \ldots, 3)$ integers. Let us define a coloring
of the \emph{edges} of the complete graph with vertices 
$\{1, 2, \ldots, N\}$ as follows:
The edge  $(i, j)$ is given the color $\chi(|i-j|)$. Ramsey's
theorem guarantees that there is a monochromatic triangle. 
Let us denote  the vertices of this
triangle by $(i,j,k)$, where $i < j <k$.
Let $a = j - i, b = k - j$, and $c = k - i$.
Then $a,b,c$ all have the same color and $a+b=c$ holds.
This gives the required monochromatic solution.

\begin{proof}[Proof of Theorem \ref{thm:fermat-euclid}]
Suppose there exist only finitely many primes $p_1, \ldots, p_k$ (say).
Every postive integer
can be written as $m= \prod_{i=1}^k p_i^{e_i}$. We write
integers as an $n$th power times an $n$th power-free number.
Hence, writing $e_i=nq_i +r_i$ with \mbox{$0 \leq r_i \leq n-1$} gives
$m= \left(\prod_{i=1}^k p_i^{q_i}\right)^n \left( \prod_{i=1}^k
  p_i^{r_i}\right) =N(m)\times R(m)$ (say).
We use $n^k$ distinct colors, denoted by $(t_1, \ldots, t_k), 0 \leq t_i \leq
n-1$, and we
color the integer $m= \prod_{i=1}^k p_i^{n q_i +r_i}$ by $(r_1, \ldots ,r_k)$. 
By Schur's theorem there exists a monochromatic
triple $(a,b,c)$ such that $c=a+b$ and with a fixed color
$(r_1, \ldots , r_k)$, corresponding to
$R=\prod_{i=1}^k p_i^{r_i}$.
Here $a,b,c$ all contain the same factor $R$ and
we can write $a,b,c$ as $a=N(a)R,\, b=N(b)R,\, c=N(c)R$, 
with positive integers $N(a), N(b), N(c)$.
 Dividing by $R$
gives $N(a)+N(b)=N(c)$ with $n$th powers, which is a contradiction to 
$\FLT(n).$

It might seem that we require unique factorization, as for an integer
with distinct prime factorizations the coloring is not well-defined.
However, for an application of Schur's theorem
it is perfectly fine if an integer $m$ with hypothetical distinct 
prime factorizations is assigned only one of the colors.
(Assigning all corresponding colors to $m$ would be an alternative, but then
$\chi$ would not actually be a function.)
\end{proof}

It is of historic interest to note that Schur's motivation was to study
Fermat's equation modulo primes. Dickson had proved 
that there is no congruence obstruction to the Fermat equation, 
and Schur \cite{Schur} gave a simple proof of this.

\section{Roth's theorem implies Euclid's theorem.}
{\label{sec:Roth}}
The Fermat equation has also been studied with coefficients.
The case $x^n+y^n=2z^n$ in positive integers
has attracted special attention, as a solution in distinct positive integers
would mean that there exist $n$th powers $x^n< z^n< y^n$ 
in arithmetic progression.
It was conjectured by D\'{e}nes that for $n\geq 3$ there exist only
trivial solutions with $x=y=z$. This was proved by 
Darmon and Merel \cite{DarmonandMerel} based on
the methods of Wiles. 
Sierpi\'nski \cite{Sierpinski}
gives elementary proofs of the cases $n=4$ (Chapter 2, $\S 8$)
and $n=3$ (Chapter 2, $\S 14$); see also \cite{Carmichael}.
We also give new proofs of Euclid's theorem in these cases.

The following result gives a matching combinatorial tool.
\begin{lemma}[Roth \cite{Roth}]
Let $\delta>0$ and $N\geq N(\delta)$.
 Every subset $ S\subset[1,N]$ of at least $\delta N$ elements contains
three distinct elements $s_1,s_2,s_3\in S$ in arithmetic progression, 
i.e., $s_1+s_3=2s_2$.
\end{lemma}
It should be noted that there is a purely combinatorial proof of Roth's theorem,
e.g., in \cite[pp. 46--49]{Graham-Rothschild-Spencer}. 
In contrast to van der Waerden's and
Schur's theorem the above statement
is a so-called ``density version'': this result not only guarantees
monochromatic solutions in some unspecified color, but even in all those 
colors that occur with a positive density.
\begin{theorem}
For $n \geq 3$ let $\DM(n)$ denote the statement
``There are no three positive $n$th powers in arithmetic progression''
or equivalently 
``There are no solutions of the equation 
$x^n+y^n=2z^n$ in positive integers $x<z<y$.''
Then ``$\DM(n)$ is true'' and Roth's theorem imply that
there exist infinitely many primes.
\end{theorem}

\begin{proof}
We first prove the following (possibly surprising) lemma.
\begin{lemma}{\label{positive-density}}
 Suppose there exist only finitely many primes $p_1< \cdots < p_k$.
The set of $n$th powers is a positive proportion of all
integers, i.e., there exists some $\delta=\delta(n,k) >0$ such that for all $N$
the set of $n$th powers in $[1,N]$ is at least $\delta N$.
\end{lemma}
\begin{proof}[Proof of Lemma]
We prove this by
dividing a lower bound approximation of the number of $n$th powers in $[1,N]$
by an upper
bound approximation of all integers in $[1,N]$, both counted by means of
exponent patterns.
The upper bound
on the number of possible exponent patterns $(e_1, e_2, \ldots ,e_k)$ follows
from $p_1^{e_1} \cdots p_k^{e_k}\leq N$, which gives $e_i \leq \frac{\log
  N}{\log p_i}$. Hence 
$(1+\frac{\log N}{\log p_1}) \cdots (1+\frac{\log N}{\log p_k})$ 
is an upper bound.
For the lower bound on the number of $n$th powers, we 
count those $e_i$ divisible by $n$ and with 
$p_i^{e_i}\leq N^{1/k}$ for all $i=1, \ldots ,k$. We see that
at least $\lfloor 1+ \frac{\log N}{nk\log p_1}\rfloor 
\cdots \lfloor 1+\frac{\log N}{nk\log p_k}\rfloor $ of all integers at most $N$
are $n$th powers, 
which gives (for large $N$) a positive proportion of at least
$\delta\geq \frac{\textrm{ lower bound}}{\textrm{upper bound}}\geq
\frac{C}{(nk)^k}$, for some $C>0$.
\end{proof}
With this lemma we can replace Schur's theorem by Roth's theorem.
Roth's theorem directly guarantees 
that there exists a nontrivial arithmetic progression of $n$th
powers, 
which is in contradiction to $\DM(n)$. 
(Note that in this case there is no need to divide by 
the factor $R$ of the first proof.)
\end{proof}

\begin{remark}
The results by van der Waerden (used by Alpoge and Granville) and Schur or Roth
(used here)
are early results of Ramsey theory. The numerical bounds on $s_t$ implied
by Schur's theorem are moderate, compared to the very quickly increasing
bounds in van der Waerden's theorem.
Let $s_t$ denote the least number such that for any $t$-coloring, which is a
map $\chi:\{1, \ldots , s_t\} \rightarrow \{1, \ldots , t\},$ there exist
$a,b,c$ with $a+b=c$ and $\chi(a)=\chi(b)=\chi(c)$. It follows from Schur's proof that
$s_t \leq \lfloor t!e\rfloor $.

\end{remark}

\section{Positive density gives a new elementary proof.}
{\label{sec:density}}

The observation about ``positive density'' in Lemma \ref{positive-density}
also leads to a short and new proof of Euclid's theorem:
\begin{proof}
  Lemma {\ref{positive-density}} says the number of
  $n$th powers (for any fixed $n \geq 2$)
  has positive density in the set of positive integers.
But it is also clear that there are at most
$N^{1/n}$  positive $n$th powers $x^n \leq N$,
contradicting the lower bound of $\delta N$ (for some fixed $\delta >0$)
for sufficiently large $N$. Comparing with
the bibliography \cite{Mestrovic}, the proof closest in spirit
appears to be Chaitin's proof \cite{Chaitin}.
\end{proof}
We note that the main focus of this article is not about short proofs 
but how seemingly remote results can be applied.

\section{Discussion on variants of the proofs above:
  the Frankl--Graham--R\"{o}dl theorem and
  Folkman's theorem.}
{\label{sec:discussion}}
\begin{enumerate}
\item 
We now discuss that knowing something more on the combinatorial side,
namely knowing about the number of monochromatic 
solutions, helps in reducing the number-theoretic input considerably.

On the combinatorial side, 
Frankl, Graham, and R\"odl \cite{Frankl-Graham-Rodl:1988} proved 
that with $t$ colors the number of monochromatic solutions 
$(a,b,c)$ of the equation $a+b=c$  with $a,b,c \in [1,N]$ increases
quadratically, i.e., there is a positive constant 
$c_t$ such that the number $S(t,N)$ of solutions is at least $c_t N^2$.
(In fact, \cite{Frankl-Graham-Rodl:1988}
gives a direct proof for the Schur equation, 
but also covers much more general cases.)
As in the proof of Theorem \ref{thm:fermat-euclid},
the monochromatic solutions of $a+b=c$
correspond to solutions of $x^n+y^n=z^n$ in positive integers.

On the number-theoretic side there are several reasons why the number of
solutions is smaller, giving a contradiction to the assumption ``there are 
finitely many primes only.''

A result of Faltings \cite{Faltings:1983} would give there are at most $O(N)$ 
solutions of $x^n+y^n=z^n$ with $x^n,y^n,z^n\in [1,N]$, being coprime in pairs.
A much more elementary approach is as follows:
For odd $n$ the left-hand side of $x^n+y^n=z^n$ can be factored as 
$(x+y)\sum_{i=0}^{n-1}(-1)^ix^{n-1-i}y^i$. In particular, when
$n=3$ this is $x^3+y^3=(x+y)(x^2-xy+y^2)$.
The number of divisors of any
integer $z^n \leq N$ is clearly at most $\sqrt{N}$. (Actually, as we assume
there are at most $k$ prime factors, this can be improved to $C_k (\log N)^k$.)

Hence the number-theoretic upper bound of at most $N$ values of $z^n$ 
with at most $\sqrt{N}$ factorizations each
and the combinatorial lower bound of at least $c_t N^2$ solutions
contradict each other.

This remark also applies in the situation of $x^n+y^n=2z^n$, as using a
result of Varnavides \cite{Varnavides:1959} one can also
prove that in this situation there would be at least $c_t N^2$ 
many solutions, with $x,y,z\leq N$, contradicting as before the 
number-theoretic upper bound.

\item{\label{Folkman}} For the combinatorial lemma there are other alternatives.
For example, a theorem of Folkman  \cite[p.~81]{Graham-Rothschild-Spencer} 
guarantees much larger monochromatic
structures than Schur's theorem does: 
For every number $t$ of colors, every
coloring $\chi: \N \rightarrow \{1, \ldots , t\}$, and every $s\in \N$, 
there exist $N_{s,t}$ and $a_1, \ldots , a_s \in [1,N_{s,t}]$
with the property that all nontrivial subset sums 
$\sum_{i \in I} a_i$, where $I\subseteq \{1, \ldots , s\}$ is nonempty,
are monochromatic.
In analogy with the proof of Theorem \ref{thm:fermat-euclid},
this would mean, in the special case $s=3$, applied with the same coloring 
and after dividing by the common factor $R$, that all of
$a'_1,a'_2,a'_3, a'_1+a'_2,a'_1+a'_3,a'_2+a'_3,a'_1+a'_2+a'_3$ are $n$th
powers. 
Proving that this is impossible could be easier than proving $\FLT(n)$,
as $\FLT(n)$ corresponds to $s=2$ with fewer conditions. But we are not
aware of any literature on this.

\end{enumerate}

\section{Hindman's theorem implies Euclid's theorem.}
{\label{sec:Hindman}}
Let us explicitly write down an extreme form of the above remark on
Folkman's theorem.
An extention of Folkman's theorem is Hindman's theorem  
\cite{Hindman:1974}; see 
also \cite{Baumgartner:1974} and 
\cite[p. 85]{Graham-Rothschild-Spencer}. 
\begin{lemma}
For any integer $t\geq 2$ and any $t$-coloring 
$\chi:\N \rightarrow \{1, \ldots , t\}$, there exists an infinite sequence 
$A=\{a_1, a_2, \ldots \}$ such that all subset sums 
$\sum_{i \in I} a_i$ over nonempty 
finite index sets $I\subset \N$ are monochromatic.
\end{lemma}
\begin{theorem}{\label{thm:Hindman-implies-Euclid}}
Hindman's theorem implies Euclid's theorem.
\end{theorem}
\begin{proof}
We start as in the proof of Theorem {\ref{thm:fermat-euclid}}.
Suppose there exist only finitely many primes $p_1, \ldots, p_k$ (say).
Every integer
can be written as $m= \prod_{i=1}^k p_i^{e_i}$, 
$e_i=nq_i +r_i$ with \mbox{$0 \leq r_i \leq n-1$}. That is, 
$m= \left(\prod_{i=1}^k p_i^{q_i}\right)^n \left( \prod_{i=1}^k
  p_i^{r_i}\right) =N(m)\times R(m)$ (say).
We
color the integer $m= \prod_{i=1}^k p_i^{n q_i +r_i}$ by $(r_1, \ldots ,r_k)$. 
By Hindman's theorem there exists an infinite set such that
all nonempty finite subset sums are monochromatic with a fixed color
$(r_1, \ldots , r_k)$, corresponding to
$R=\prod_{i=1}^k p_i^{r_i}$.
 Dividing by $R$
gives an infinite set such that all finite subset sums
are $n$th powers.

This would
in particular correspond to some fixed $x^n$ and infinitely many 
pairs $(y_i^n,z_i^n)$ of $n$th powers such that $x^n+y_i^n=z_i^n$ holds.
This is clearly impossible, as the difference between
consecutive $n$th powers $z^n-(z-1)^n\geq z^{n-1}$ increases when $n\geq 2$ 
is fixed and $z$ increases. 
\end{proof}
\begin{remark}
The proof of Hindman's theorem is not trivial, but it is certainly much more
accessible than $\FLT(n)$ for general $n$. Moreover, the proof of Hindman's
theorem does not make use of Euclid's theorem, in contrast to Wiles's proof of
$\FLT$.

\end{remark}

\section{Infinite Ramsey theory implies Euclid's theorem.}
{\label{sec:IRT}}
The above proof does not need the full strength of Hindman's theorem, as
it essentially only uses sums of two elements.
Hence it is possible to reduce the combinatorial input accordingly, which we
discuss below.
\begin{lemma}[The Infinite Ramsey Theorem $\IRT$, see 
e.g., {\cite[Theorem 9.1.2]{Diestel}}]\label{lemma:IRT}

Let $X$ be some infinite set and color all subsets of $X$ of size $w$
with $t$ different colors. Then there exists some infinite subset
$M\subset X$ such that the subsets of $M$ of size $w$ all have the same color.
\end{lemma}
In plain words, the case $w=2$ of Lemma \ref{lemma:IRT}
says that a finite coloring of the complete
graph $K_{\infty}$ guarantees a complete monochromatic $K_{\infty}$ as a
subgraph.
\begin{theorem}{\label{thm:IRT-implies-Euclid}}
The infinite Ramsey theorem $\IRT$ implies Euclid's theorem.
\end{theorem}
We leave the proof of Theorem \ref{thm:IRT-implies-Euclid}
as an exercise to the reader, and only
remark it is a variant of the Theorem
{\ref{thm:Hindman-implies-Euclid}} and our final Theorem
{\ref{thm:K2inf-implies-Euclid}}.

It turns out that one does not actually need an
infinite complete monochromatic graph, but only
a monochromatic complete bipartite graph $K_{2, \infty}$,
where one set of the vertices 
consists of two elements and the other one is infinite (say countable).

We give a complete proof of this and the application to Euclid's theorem below.
To prove the existence of this infinite substructure is quite simple.
\begin{lemma}[The $\boldsymbol K_{2, \infty}$ Lemma]
Let $X$ be some infinite set and color all pairs of two distinct elements
of $X$
with $t$ different colors. Then there exist a set $V=\{v_1,v_2\}\subset X$ and
an infinite set $W=\{w_1,w_2, \ldots \}\subset X\backslash V$ such that all
edges $(v_i,w_j)$, with $i\in \{1,2\}$ and $j \in \N$, have the same color.
\end{lemma}
For ease of notation we assume that $X$ is countable.
\begin{proof}
One can construct the required sets step by step.

Choose any set $A=\{a_1, a_2, \ldots , a_{t+1}\}\subset X$ of $t+1$ distinct 
elements as vertices.
Let $v_1=a_1$. There are infinitely many adjacent edges $(v_1, x_j)$.
Hence one
of the $t$ colors, say color $c_1$, occurs infinitely often.
Let $X_1=\{x_{1,j}:j \in \N\} \subset X$ be the set of those elements such that
$(v_1, x_{1,j})$ are these infinitely many edges of color $c_1$.
Now study the color of all $(a_i, x_{1,j})$ as follows.
There exists one color $c_2$ (say) that occurs infinitely often among 
the infinitely many edges  $(a_2, x_{1,j})$. Let $X_2=\{x_{2,j}:j \in \N\} 
\subset X_1$ be those elements such that  $(a_2, x_{2,j})$ are of color $c_2$.
If $c_1=c_2$ we have found the
required substructure with $V=\{a_1,a_2\}$ and $W=X_2$.
We therefore assume that $c_1\neq c_2$.
We iterate the step above and come to infinite subsets
$X_{t+1} \subset X_t \subset \cdots \subset X_3\subset X_2 \subset X_1 
\subset X$ such that for fixed $i$ all 
edges $(a_i, x_{i,j}), j \in \N$, are of color $c_i$ (say).
As there are $t$ distinct colors only, there must be two distinct 
indices $i_1,i_2\in \{1, \ldots ,t+1\}$ such
that $c_{i_1}=c_{i_2}$. With $i_1 <i_2$ without loss of generality
and $V=\{a_{i_1}, a_{i_2}\}, W=X_{i_2}$ and the lemma is proved.

An alternative is to color the elements $x \in X\backslash A$ 
with the vector color $(c_1, \ldots, c_{t+1})$ if the color of the edge
$(a_i,x)$ is $c_i$, $i=1, \ldots , t+1$. 
As there is only a finite number of vector colors, namely
 $t^{t+1}$, there is an infinite
number of $x \in X\backslash A$ with the same vector color, which 
defines the set
$W$. As before, there
are two indices $i_1 \neq i_2$ such that $c_{i_1}=c_{i_2}$. Hence 
$V=\{a_{i_1}, a_{i_2}\}$ and $W$ are the sets required.
\end{proof}
\begin{theorem}{\label{thm:K2inf-implies-Euclid}}
The $K_{2,\infty}$ lemma implies Euclid's theorem.
\end{theorem}

\begin{proof}
Let $n\geq 2$, and assume that $p_1, \ldots , p_k$ is the list of all primes.
 We color the integers by the same rule as before:
$m= \prod_{i=1}^k p_i^{n q_i +r_i}$ is colored by $\chi(m)=(r_1, \ldots ,r_k)$.
Based on this coloring we define an infinite graph on the positive integers. 
The edges $(m_i,m_j)$ receive the color $\chi(m_i+m_j)$.

We apply the $K_{2,\infty}$ lemma to this graph: 
there exists a complete bipartite graph with parts $V=\{v_1,v_2\}$ and an
infinite set $W$
 such that all
 edges $(v_i,w_j)$, with $i\in \{1,2\}$ and $j \in \N$, have the same color
 $(r_1, \ldots ,r_k)$.

We multiply all integers in $\N$ by the constant
$P=\prod_{i=1}^k p_i^{n-r_i}$.
All pairwise sums
$Pv_i+Pw_j=P(v_i+w_j)$ are an $n$th power $z_{i,j}^n$ (say). Note
that $z_{2,j}^n-z_{1,j}^n=P(v_2-v_1)$ is a constant, 
and is also the distance between infinitely many distinct
pairs of $n$th powers, for the infinitely many values $j$.
This is impossible, as
the gap between consecutive $n$th powers increases (see above).
\end{proof}

With Hindman's theorem we
made use of a quite advanced combinatorial result, 
and the number-theoretic part became correspondingly quite simple.
We then reduced the depth of the combinatorial lemma until we reached 
the $K_{2, \infty}$ lemma.
On the number-theoretic side, we eventually used the elementary
fact that the gaps between consecutive $n$th powers increase and simple
arithmetic such as $P(m_i+m_j)=Pm_i+Pm_j$.

\section{Conclusion.}
 
As our journey through a fictional world comes to an end, 
let us briefly reflect:
a common theme in all variants discussed is that the existence of only 
finitely many
primes would guarantee patterns for the set of $n$th powers that cannot
actually exist, sometimes for deep reasons, 
sometimes for obvious ones, depending on
the strength of the pattern.
Summarizing the results we find:
\begin{corollary}
In the ``world with only finitely many primes'' the following hold:
\begin{enumerate}
\item
If Schur's theorem holds, then $\FLT(n)$ is wrong for all $n \geq 3$.\\
If $\FLT(n)$ holds for some $n \geq 3$, then Schur's theorem does not hold. 
\item If Roth's theorem holds, then $\DM(n)$ is wrong for all $n \geq 3$.\\
If $\DM(n)$ holds for some $n \geq 3$, then Roth's theorem does not hold. 
\item The set of $n$th powers has positive density
  (giving an immediate contradiction).
\item Hindman's theorem does not hold.
\item The infinite Ramsey theorem (IRT) does not hold.
\item The $K_{2, \infty}$ lemma does not hold.
\end{enumerate}
\end{corollary}
In other words, Euclid's theorem is logically connected
with many interesting and seemingly unrelated results in mathematics.

Having seen all these variants and extensions, the original version, 
i.e., the combination of Schur's
theorem and the Fermat--Wiles theorem is the one that
looks most intriguing to this author. And Fermat's last theorem
may be the one that many of us
would miss most in the fictional ``world with only finitely many primes''!

\begin{acknowledgment}{Acknowledgments.}
The author would like to thank the referees, the editor,
R.~Dietmann, 
J.~Erde, I.~Leader, R.~Me\v{s}trovi\'{c},
J.-C.~Schlage-Puchta and A.~Wiles
for useful comments on the manuscript.
The author was partially supported by the Austrian Science Fund (FWF): W1230
and I 4945-N.
\end{acknowledgment}

  
\begin{affil}
Christian Elsholtz, Institut f\"ur Analysis und Zahlentheorie,
Technische Universit\"at Graz,
Kopernikusgasse 24,
A-8010 Graz, Austria.\\
{\it elsholtz@math.tugraz.at}
\end{affil}

\end{document}